\documentclass[12pt,a4paper]{amsart}
\usepackage{a4wide}
\usepackage[english,russian]{babel}
\usepackage[T2A]{fontenc}
\usepackage[cp1251]{inputenc} 
\usepackage{amsfonts}
\usepackage{amssymb, amsthm, amscd}
\usepackage{amsmath}
\usepackage{mathtools}
\usepackage{needspace}
\usepackage{etoolbox}
\usepackage{lipsum}
\usepackage{comment}
\usepackage{cmap}
\usepackage[pdftex]{graphicx}
\usepackage[unicode]{hyperref}
\usepackage[matrix,arrow,curve]{xy}
\usepackage[usenames,dvipsnames]{xcolor}
\usepackage{colortbl}
\usepackage{textcomp}
\usepackage{cite}
\usepackage{euscript}

\DeclareMathOperator{\pr}{pr}
\DeclareMathOperator{\id}{id}

\DeclareMathOperator{\aug}{aug}

\newcommand{\Z}{\mathbb{Z}}

\newcommand{\N}{\mathbb{N}}

\newcommand{\M}{\mathfrak{M}}

\newcommand{\fr}[1]{\mathfrak{#1}}

\newcommand{\ovl}{\overline}

\newcommand{\eps}{\varepsilon}
\newcommand{\ph}{\varphi}

\newcommand{\sub}{\subseteq}
\newcommand{\Hom}{\mathop{\mathrm{Hom}}\nolimits}
\newcommand{\Ker}{\mathop{\mathrm{Ker}}\nolimits}

\newcommand{\Spec}{\mathop{\mathrm{Spec}}\nolimits}
\newcommand{\Max}{\mathop{\mathrm{Max}}\nolimits}
\newcommand{\rk}{\mathop{\mathrm{rk}}\nolimits}

\renewcommand{\ge}{\geqslant}
\renewcommand{\le}{\leqslant}
\newcommand{\sm}{\setminus}
\newcommand{\map}[3]{#1\colon #2\to #3}
\newcommand{\phan}{\phantom{,}}

\newcommand{\ox}{\otimes}

\newcommand{\<}{\langle}
\renewcommand{\>}{\rangle}

\newcommand{\euP}{\EuScript{P}}
\newcommand{\euQ}{\EuScript{Q}}

\newcommand\leftact[2]{\phan^{#1}#2}

\theoremstyle{plain}
\newtheorem{thm}{Theorem}
\newtheorem{lem}{Lemma}
\newtheorem{prop}{Proposition}
\newtheorem{cor}{Corollary}

\theoremstyle{definition}

\theoremstyle{remark}
\newtheorem*{rem*}{Remark}
\newtheorem*{rems}{Remarks}

\AtBeginEnvironment{thm}{\begin{samepage}}
\AtEndEnvironment{thm}{\end{samepage}}
\AtBeginEnvironment{lem}{\begin{samepage}}
\AtEndEnvironment{lem}{\end{samepage}}
\AtBeginEnvironment{st}{\begin{samepage}}
\AtEndEnvironment{st}{\end{samepage}}
\AtBeginEnvironment{crit}{\begin{samepage}}
\AtEndEnvironment{crit}{\end{samepage}}
\AtBeginEnvironment{ax}{\begin{samepage}}
\AtEndEnvironment{ax}{\end{samepage}}
\AtBeginEnvironment{defn}{\begin{samepage}}
\AtEndEnvironment{defn}{\end{samepage}}
\AtBeginEnvironment{cor}{\begin{samepage}}
\AtEndEnvironment{cor}{\end{samepage}}
\AtBeginEnvironment{note}{\begin{samepage}}
\AtEndEnvironment{note}{\end{samepage}}
\AtBeginEnvironment{prop}{\begin{samepage}}
\AtEndEnvironment{prop}{\end{samepage}}

\newcommand{\Stab}{\mathop{\mathrm{Stab}}\nolimits}
\newcommand{\SL}{\mathop{\mathrm{SL}}\nolimits}
\newcommand{\Sp}{\mathop{\mathrm{Sp}}\nolimits}

\newcommand{\tc}{\text{,}}
\newcommand{\tp}{\text{.}}

\renewcommand{\tilde}{\widetilde}
\renewcommand{\hat}{\widehat}

\DeclareMathOperator{\gen}{gen}

\makeatletter
\def\@settitle{\begin{center}%
    \baselineskip14\p@\relax
    \bfseries
    \@title
  \end{center}%
}

\def\@evenhead{\hfil\sc p. gvozdevsky\hfil}
\def\@oddhead{\hfil\sc overgroups of subsystem subgroups \hfil}
\makeatother

\title{Overgroups of subsystem subgroups in exceptional groups:\\ inside a sandwich}

\address{Chebyshev Laboratory, St. Petersburg State University, 14th Line V.O., 29, Saint Petersburg 199178 Russia.}

\thanks{The author is a participant of a scientific group that won "Leader" grant by "BASIS" foundation in 2020, grant \#20-7-1-27-4. Research is also supported by «Native towns», a social investment program of PJSC ''Gazprom Neft'', and also by grant given as subsidies from Russian federal budget for creation and development of international mathematical centres, agreement between MES and PDMI RSA  \textnumero 075-15-2019-1620 from November 8, 2019.}

\author{P.~Gvozdevsky}
\email{gvozdevskiy96@gmail.com}
\date{}


\begin{document}
\selectlanguage{english}

\maketitle

\begin{abstract}
	The current paper is an addition to the previous paper by author ({\it Overgroups of subsystem subgroups in exceptional groups\textup: a $2{A}_1$\nobreakdash-proof}, (2020)), where the overgroup lattice of the elementary subsystem subgroup $E(\Delta,R)$ of the Chevalley group $G(\Phi,R)$ for a large enough root subsystem $\Delta$ was studied. Now we study the connection between the elementary subgroup $\hat{E}(\sigma)$ given by the net of ideals~$\sigma$ of the ring~$R$, and the stabilizer $S(\sigma)$ of the corresponding Lie subalgebra of the Chevalley algebra. In particular, we prove that under a certain condition the subgroup $\hat{E}(\sigma)$ is normal in $S(\sigma)$, and we also study some properties of the corresponding quotient group.
\end{abstract}

\section{Introduction}

In the paper \cite{VavStepSurvey} its authors suggested the following problem: to describe overgroups of the elementary subsystem subgroup $E(\Delta,R)$ in the Chevalley group $G(\Phi,R)$ for a large enough subsystem~$\Delta$; the corresponding hypotheses were formulated in \cite{Vsch}. According to these hypotheses the description must be given in terms of nets of ideals of the ground ring $R$ (see definition in Subsection~\ref{nets}). Namely, for any overgroup $E(\Delta,R)\le H\le G(\Phi,R)$ the ground ring $R$ has a unique net of ideals $\sigma$ called a {\it level} of this overgroup such that $E(\sigma)\le H\le N_{G(\Phi,R)}(E(\sigma))$, where the group $E(\sigma)$ is defined as in Subsection~\ref{NetSubgroups}. Such answer is called a {\it standard description} or a {\it sandwich-classification}. Here by sandwich we mean the set of overgroups of some fixed level.

Most of that hypotheses with certain rectification were proved in the author's paper~\cite{Gvoz2A1} assuming that the ring $R$ does not have a residue field of two elements. In that paper the known results for classical groups are not just generalised to exceptional groups but also are rectified in such a way to take into account irreducible components of type~$A_1$ in the subsystem~$\Delta$. The rectification was that the normaliser of the group $E(\sigma)$ was replaced by the stabiliser $S(\sigma)$ of certain subalgebra $L(\sigma)$ in Chevalley algebra $L(\Phi,R)$ (see definitions in Subsections~\ref{LieAlgebras} and~\ref{NetSubgroups}). Such form of the answer was called a {\it pseudostandard description}. 

It is easy to see that under certain assumption the inclusion $N_{G(\Phi,R)}(E(\sigma))\le S(\sigma)$ holds true (see Corollary 2 in \cite{Gvoz2A1}). Therefore, the question on normality of $E(\sigma)$ in $S(\sigma)$ arises; this question is, in fact, independent to results of \cite{Gvoz2A1}. 

There are many interesting cases where the subsystem $\Delta$ has irreducible components of type $A_1$; in these cases we do not have normality in general because the elementary group $E(2,R)$ is not normal in $\SL(2,R)$ in general. In order to rectify that, we define the groups $\hat{E}(\Delta,R)$ and $\hat{E}(\sigma)$ by adding certain unipotents to the set of generators. If for a given subsystem and a given ring the pseudostandard description in sence of~\cite{Gvoz2A1} holds true, then for overgroups of the subgroup $\hat{E}(\Delta,R)$ the similar description with $E(\sigma)$ replaced by $\hat{E}(\sigma)$ holds true. In addition,  in case where the subsystem $\Delta$ has no irreducible components of type $A_1$, we have $\hat{E}(\Delta,R)=E(\Delta,R)$ and $\hat{E}(\sigma)=E(\sigma)$. 

In the present paper we prove under certain assumptions (see condition~$(*)$ in \S\ref{statements}) that the subgroup $\hat{E}(\sigma)$ is normal in $S(\sigma)$ (Theorem \ref{normal}). Our proof uses the same methods as the proof of normality of the elementary subgroup in Chevalley group given by Taddei in~\cite{Taddei}. Before we prove Theorem~\ref{normal} we will need to study the intersection of the group $S(\sigma)$ with the congruence subgroup whose level is the Jacobson radical (Theorem \ref{Jacobson}); also we will need to write down the generators of the group $S(\sigma)$ in case where the ground ring is local (Theorem \ref{local}). 

Further we define the group $G(\sigma)$, which is something like the ''connected component'' of the group $S(\sigma)$. Theorem~\ref{finiteindex} of the present paper shows that this subgroup is normal in $S(\sigma)$ and the corresponding quotient group can be embedded in a product of certain subquotients of the Weyl group $W(\Phi)$. In case where the ground ring is Noetherian the number of factors is finite; hence the quotient group $S(\sigma)/G(\sigma)$ is also finite. For the classical groups these results are known from the papers by Borevich, Vavilov and Shchegolev \cite{BV84, VavMIAN, VavSplitOrt, SchegDiss, SchegMainResults,  SchegSymplectic}.

Another result of the present paper is similar to the theorem proved in~\cite{HazVav} about the $K_1$ group being nilpotent. Namely, assuming, in addition, that the subsystem $\Delta$ has no irreducible components if type $A_1$, we prove that the group $G(\sigma)/E(\sigma)$ is nilpotent-by-abelian (Theorem~\ref{nilpotent}). Note that the special case of this Theorem was proved in the author's paper~\cite{GvozLevi}, where the problem on overgroups for subsystems $A_l\le D_l$, $D_5\le E_6$ and $E_6\le E_7$ was studied.

We introduce all the necessary notation in \S\ref{notation}; after that in \S\ref{statements} we give precise statements for all the theorems mentioned above. The remaining Sections (\S\S4--10) contain proofs.

\section{Basic notation}\label{notation}

\subsection{Root systems and Chevalley groups}

Let $\Phi$ be an irreducible root system in the sense of \cite{Bourbaki4-6}, $\euP$ a lattice that is intermediate between the root lattice $\euQ(\Phi)$ and the weight lattice $\euP(\Phi)$, $R$ a commutative associative ring with unity, $G(\Phi,R)=G_{\euP}(\Phi,R)$ a Chevalley group of type $\Phi$ over $R$, $T(\Phi,R)=T_{\euP}(\Phi,R)$ a split maximal torus of $G(\Phi,R)$. For every root $\alpha\in\Phi$ we denote by $X_\alpha=\{x_\alpha(\xi),\colon \xi\in R\}$ the corresponding unipotent root subgroup with respect to $T$. We denote by $E(\Phi,R)=E_{\euP}(\Phi,R)$ the elementary subgroup generated by all $X_\alpha$, $\alpha\in\Phi$. 

Let $\Delta$ be a proper subsystem of $\Phi$. We denote by $E(\Delta,R)$ the subgroup of $G(\Phi,R)$, generated by all~$X_\alpha$, where $\alpha\in \Delta$. It is called an (elementary) {\it subsystem subgroup}. It can be shown that it is an elementary subgroup of a Chevalley group $G(\Delta,R)$, embedded into the group $G(\Phi,R)$. Here the lattice between $\euQ(\Delta)$ and $\euP(\Delta)$ is an orthogonal projection of $\euP$ onto the corresponding subspace.

We denote by $W(\Phi)$ resp. $W(\Delta)$ the Weyl groups of the systems $\Phi$ resp. $\Delta$.

\subsection{Group theoretic notation}
\begin{itemize}
	
	\item If the group $G$ acts on the set  $X$, $x\in X$ and $Y\sub X$, we denote by $\Stab_G(x)$ resp. $\Stab_G(Y)$ the stabiliser
	of the element~$x$ resp. the stabiliser of the subset $Y$ (as a subset, not pointwise).
	
	\item Commutators are left normalised:
	$$
	[x,y]=xyx^{-1}y^{-1}.
	$$
	
	\item Upper index stands for the left or right conjugation: 
	$$
	\leftact{g}{h}=ghg^{-1},\quad h^g=g^{-1}hg.
	$$
	
	\item If $X$ is a subset of the group $G$, we denote by $\<X\>$ the subgroup generated by $X$.
	
	\item If $H\le G$, then by $H^G$ we denote the normal closure of the subgroup $H$ in the group $G$.  
\end{itemize}

\subsection{Reduction homomorphisms and congruence subgroups}

If $R$ is a ring and $I\unlhd R$ is an ideal, then the projection onto the quotient $R\to R/I$, and also the corresponding reduction homomorphism $G(\Phi,R)\to G(\Phi,R/I)$ we denote by $\rho_I$.

The kernel of this homomorphism is called the {\it principal congruence subgroup} of level $I$; we denote it by $G(\Phi,R,I)$.

We also denote $T(\Phi,R,I)=G(\Phi,R,I)\cap T(\Phi,R)$.

The {\it elementary congruence subgroup} of level $I$ is the group
$$
E(\Phi,R,I)=\<\{x_\alpha(\xi)\colon \alpha\in\Phi,\, \xi\in I\}\>^{E(\Phi,R)}\tp
$$

\subsection{Nets of ideals}\label{nets}

We need the following definition; see, for example, sur\-vey~\cite{VavSbgs} and further references in it.

A collection of ideals $\sigma=\{\sigma_\alpha\}_{\alpha\in\Phi}$ of the ring $R$ is called a {\it net of ideals}, if the following conditions hold true.
\begin{enumerate}
	\item If $\alpha$, $\beta$, $\alpha+\beta\in\Phi$, then $\sigma_\alpha\sigma_\beta\sub \sigma_{\alpha+\beta}$.
	
	\item If $\alpha\in\Delta$, then $\sigma_\alpha=R$.
\end{enumerate} 

The next lemma follows directly from the definition. 

\begin{lem}
	\label{EqualIdeals} If $\sigma$ if a net of ideals\textup, and $\alpha,\beta\in\Phi$ are such that $\alpha-\beta\in\Delta,$ then $I_\alpha=I_\beta$.
\end{lem}

\medskip

If $I\unlhd R$ is an ideal, and $\map{f}{R}{A}$ is a ring homomorphism, then by $f(I)$ we denote the ideal in  $A$ generated by the image of the ideal $I$. If $\sigma$ is a net of ideals in $R$, then we set $f(\sigma)_\alpha=f(\sigma_\alpha)$, which is a net of ideals in $A$.

Given two nets of ideals $\sigma_1$,$\sigma_2$ in the ring $R$, we write $\sigma_1\le\sigma_2$ if $(\sigma_1)_\alpha\le(\sigma_2)_\alpha$ for any $\alpha\in\Phi$.

\subsection{Lie algebras}\label{LieAlgebras}

We denote by $L(\Phi,\Z) $ the integer span of the Chevalley basis in the complex Lie algebra of type $\Phi$ (see~\cite{Humphreys}). The following notation stands for Chevalley algebra
$L(\Phi,R)=L(\Phi,\Z)\ox_\Z R$. This is a Lie algebra over the ring $R$ equipped with an action of the group $G(\Phi,R)$ called the adjoint representation. For elements $g\in G(\Phi,R)$ and $v\in L(\Phi,R)$, we denote this action by $\leftact{g}{v}$.

Note that the algebra $L(\Phi,R)$, in general, is not isomorphic to the tangent Lie algebra of the algebraic group $G(\Phi,R)$ (see, for example, \cite{RoozemondDiss}). However, if the group is simply connected, then these algebras are canonically isomorphic.


We denote by $e_\alpha$, $\alpha\in\Phi$ and $h_i$, $i=1,\ldots,\rk\Phi$ the Chevalley basis of the Lie algebra $L(\Phi,R)$. By $D$ we denote its toric subalgebra generated by~$h_i$.

For every element $v\in L(\Phi,R)$, we denote by $v^\alpha$ and $v^i$ its coefficient in Chevalley basis. 

For a net of ideals~$\sigma$ we set
$$
L(\sigma)=D\oplus\bigoplus_{\alpha\in\Phi}\sigma_\alpha e_\alpha\le L(\Phi,R),
$$

The next lemma follows directly from the definition of a net. 

\begin{lem}\label{lie} The submodule $L(\sigma)$ is a Lie-subalgebra in $L(\Phi,R)$.
\end{lem}

\subsection{Net subgroups}\label{NetSubgroups}

For a net of ideals~$\sigma$ we set
$$
E(\Phi,\Delta,R,\sigma)=\<x_\alpha(\xi)\colon \alpha\in\Phi,\, \xi\in\sigma_\alpha\>,\quad
S(\Phi,\Delta,R,\sigma)=\Stab_{G(\Phi,R)}(L(\sigma))\tp
$$
For simplicity, when the other parameters are clear from the context, we write just $E(\sigma)$ and $S(\sigma)$.

\subsection{Images of $\SL_2$}

For a root $\alpha\in \Phi$ we denote by $\ph_{\alpha}$the homomorphism of algebraic groups
$$
\map{\ph_{\alpha}}{\SL(2,R)}{G(\Phi,R)}
$$
such that
$$
\ph\left(\begin{pmatrix} 1 & \xi \\ 0 & 1\end{pmatrix}\right)=x_{\alpha}(\xi) \text{ and } \ph\left(\begin{pmatrix} 1 & 0 \\ \xi & 1\end{pmatrix}\right)=x_{-\alpha}(\xi)\tp
$$
For $\eps\in R^*$ we set 
$$
h_\alpha(\eps)=\ph_{\alpha}\left(\begin{pmatrix} \eps & 0\\ 0 & \eps^{-1} \end{pmatrix} \right)\tp
$$
For $\xi$,$\zeta$,$\eta\in R$ set
$$
t_{\alpha}^{\zeta,\eta}(\xi)=\ph_{\alpha}\left(\begin{pmatrix} 1-\xi\zeta\eta & \xi\zeta^2\\ -\xi\eta^2 & 1+\xi\zeta\eta  \end{pmatrix}\right)\tp
$$
The matrix in the last notation is the transvection
$$
e+\begin{pmatrix}\zeta\\ \eta\end{pmatrix}\xi\begin{pmatrix}-\eta & \zeta\end{pmatrix}\tp
$$

Let us make the following observation.
\begin{lem}\label{ConjugeteTransvection} Let $\alpha\in\Phi,$ $g\in \SL(2,R),$ $\xi,\zeta,\eta\in R,$ and let
	$$
	g\cdot \begin{pmatrix} \zeta \\ \eta\end{pmatrix}=\begin{pmatrix} \zeta_1 \\ \eta_1\end{pmatrix}\tp
	$$
	Then $\leftact{\ph_{\alpha}(g)}{t_{\alpha}^{\zeta,\eta}(\xi)}=t_{\alpha}^{\zeta_1,\eta_1}(\xi)$.
\end{lem}

\medskip

For every net of ideals $\sigma$ we set
\begin{align*} 
	\hat{E}(\sigma)&=\hat{E}(\Phi,\Delta,R,\sigma)
	\\
	&=\big\<\{x_\alpha(\xi)\,:\, \alpha\in\Phi,\, \xi\in\sigma_\alpha\}\cup\{t_{\alpha}^{\zeta,\eta}(\xi)\,:\, \alpha\in\Delta,\;\xi,\zeta,\eta\in R\}\big\>\tp
\end{align*} 

\begin{rems}
	\begin{enumerate}
		\item If a root $\alpha$ belongs to a component of the subsystem $\Delta$ with rank at least~2, then the generators $t_{\alpha}^{\zeta,\eta}(\xi)$ are superfluous because they can be expressed by elementary root elements, see Lemma~\ref{notA1}.  However, it was shown by Cohn~\cite{CohnGL2} that this does not hold for the group $\SL(2,R)$ over the polynomial ring of two variables with coefficients in a field.
		
		\item If we set 
		$$
		\quad\qquad\hat{E}(\Delta,R)=\big\<\big\{x_\alpha(\xi)\,:\, \alpha\in\Delta,\, \xi\in R\big\}\cup\{t_{\alpha}^{\zeta,\eta}(\xi)\,:\, \alpha\in\Delta,\;\xi,\zeta,\eta\in R\}\big\>,
		$$
		then we have $\hat{E}(\sigma)=\<E(\sigma)\cup \hat{E}(\Delta,R)\>$. Therefore, if for a given sub\-system and a given ring the pseudostandard description of overgroups in sense of~\cite{Gvoz2A1} holds, then for overgroups of the subgroup $\hat{E}(\Delta,R)$ the similar description holds with $E(\sigma)$ being replaced by $\hat{E}(\sigma)$. 
		
		\item The definition of the group $\hat{E}(\sigma)$ is motivated by the fact that we can prove its normality the group $S(\sigma)$. The subgroup $E(\sigma)$ is not normal in $S(\sigma)$ in general because the elementary group is not normal in $\SL(2,R)$ in general.
		
		N.~A.~Vavilov suggested that the results of the present paper remain true if instead of using all the elements $t_{\alpha}^{\zeta,\eta}(\xi)$ one use only the images of elements from the normal closure of elementary subgroup in $\SL(2,R)$. However such set of generators is not so good when interacting with localisation; hence our proof would not work for it; the problem arises in the first item of Lemma~\ref{denominators}.
		
		\item  Elements $t_{\alpha}^{\zeta,\eta}(\xi)$ are images of symplectic transvections (recall that $\SL(2,R)=\Sp(2,R)$). N.~A.~Vavilov also suggested that the subgroup of $\SL(2,R)$ generated by them coincides with the subgroup generated by all transvections; in other words, that Eichler group for $\Sp(2,R)$ coincides with Eichler group for $\SL(2,R)$. However, we do not know any such results.
	\end{enumerate}
\end{rems}

\begin{lem} Let $\sigma$ be a net of ideals. Then $\hat{E}(\sigma)\le S(\sigma)$.
\end{lem}

\begin{proof}
	Direct calculation.
\end{proof}

\subsection{Localisation}

Let $s\in R$ be a non nilpotent element. Let $R_s=R[s^{-1}]$ be the principal localisation. We denote by $F_s$ the natural localisation ho\-mo\-mo\-rphism $\map{F_s}{R}{R_s}$, and also the corresponding group homomorphism $\map{F_s}{G(R)}{G(R_s)}$.

Similarly, if $\fr{P}$ be a prime ideal, then we denote by $R_{\fr{P}}$ the localisation of the ring at this ideal; by $F_{\fr{P}}$ we denote the corresponding homomorphism.

We need the following lemma.

\begin{lem}
	\label{LocalisationForS} Let $\sigma$ be a net of ideals in $R$. Then
	$$
	S(\sigma)=\bigcap_{\M\in\Max(R)} F_{\M}^{-1}\big(S(F_{\M}(\sigma))\big),
	$$
	where $\Max(R)$ is the set of all maximal ideals in $R$.
\end{lem}

\begin{proof}
	Inclusion of the left-hand side to the right-hand side is obvious. Further it is easy to see that $$
	F_{\M}^{-1}\big(S(F_{\M}(\sigma))\big)\le S\big(F_{\M}^{-1}(F_{\M}(\sigma))\big)
	$$
	for any $\M\in\Max(R)$; and also that
	$$
	L(\sigma)=\bigcap_{\M\in\Max(R)} L\big(F_{\M}^{-1}(F_{\M}(\sigma))\big)\tc
	$$
	which implies the converse inclusion.
\end{proof}

\subsection{Affine schemes and generic elements}

The functor $G(\Phi,-)$ from the category of rings to the category of groups is an affine group scheme (a Chevalley--Demazure scheme). This means that its composition with the forgetful functor to the category of sets is representable, i.e.,
$$
G(\Phi,R)=\Hom (\Z[G],R).
$$
The ring $\Z[G]$ here is called the {\it ring of regular functions} on the scheme $G(\Phi,-)$. 

An element $g_{\gen}\in G(\Phi,\Z[G])$ that corresponds to the identity ring ho\-mo\-mo\-rphism is called the {\it generic element} of the scheme $G(\Phi,-)$. This element has a universal property: for any ring $R$ and for any $g\in G(\Phi,R)$, there exists a unique ring homomorphism
$$
\map{f}{\Z[G]}{R},
$$
such that $f_*(g_{\gen})=g$. For details about application of the method of generic elements
to the problems similar to that of ours, see the paper of A.~V.~Ste\-pa\-nov~\cite{StepUniloc}.

Further for a commutative ring $R$ we set 
$$
R[G]=R\otimes_{\Z} \Z[G].
$$
The image of the element $g_{\gen}$ under the group homomorphism induced by the natural ring homomorphism $\Z[G]\to R[G]$ we denote by 
$$
g_{\gen,R}\in G(R[G]).
$$
This element also has a universal property: for any $R$\nobreakdash-algebra $A$ and any $g\in G(\Phi,A)$ there exists a unique $R$-algebra homomorphism 
$
\map{f}{R[G]}{A}
$
such that $f_*(g_{\gen,R})=g$. 

In particular, {\it augmentation homomorphism} is a unique $R$-algebra ho\-mo\-mo\-rphism 
$
\map{i}{R[G]}{R}
$
such that $i_*(g_{\gen,R})=e$. It's kernel
$$
I_{\aug,R}=\Ker i
$$
is called the {\it augmentation} ideal.

If $\xi\in R[G]$ and $g\in G(\Phi,A)$ for some $R$-algebra $A$, then we set $\xi(g)=f(\xi)$, where $f$ is the $R$-algebra homomorphism as above. Therefore, elements of the ring $R[G]$ can be considered as functions. 

Note that if $\{\sigma_i\}$ is a family of ideals in some ring, then we have
$$
\bigcap_i S(\sigma_i)\le S\Big(\bigcap_i \sigma_i\Big)\tp
$$ 
This allows us to introduce the following notation. Let $\sigma$ be a net of ideals in the ring $R$. Denote by $\dot{\sigma}$ the smallest net of ideals in $R[G]$ such that $g_{\gen,R}\in S(\dot{\sigma})$ and for any $\alpha\in\Phi$ we have $\sigma_\alpha\otimes 1\sub \dot{\sigma}_\alpha$. It can be obtained as intersection of all nets with such properties. 

We need the following lemma.

\begin{lem}\label{Sgeneric} Let $\sigma$ be a net of ideals in the ring $R$; and let $g\in S(\sigma)$. Let $\map{f}{R[G]}{R}$ be such $R$-algebra homomorphism that $f_*(g_{\gen,R})=g$. Then $f(\dot{\sigma})\le\sigma$.
\end{lem}

\begin{proof}
	It is easy to see that $f^{-1}(\sigma)$ is a net of ideals in $R[G]$ such that $g_{\gen,R}\in S(f^{-1}(\sigma))$ and for any $\alpha\in\Phi$ we have
	$$
	\sigma_\alpha\otimes 1\sub f^{-1}(\sigma)_\alpha.
	$$
	Hence $\dot{\sigma}\le f^{-1}(\sigma)$, which means exactly that $f(\dot{\sigma})\le\sigma$.
\end{proof}

\section{Statements of the main results}\label{statements}

In the rest of the paper, in particular in all the theorems below, we assume that the following condition holds true. 
\begin{equation}
	\begin{split}
		&\text{For any }\gamma\in\Phi\sm\Delta\text{ there exists }\alpha\in\Delta\text{ such that }\alpha+\gamma\in\Phi\tc\\& \text{and the corresponding }
		\text{structural constant } c_{\alpha,\gamma} \text{ of the Chevalley algebra}\\ & \text{is invertible in the ring } R\tp
	\end{split}\tag{$*$}
\end{equation}

If $\sigma$ is a net of ideals in $R$ and $I\unlhd R$, then we set $(\sigma\cap I)_\alpha=\sigma_\alpha\cap I$. The collection of ideals $\sigma\cap I$ satisfies only the first condition in the definition of a net; nevertheless, we consider the subgroup
$$
E(\sigma\cap I)=\<x_\alpha(\xi)\colon \alpha\in\Phi,\, \xi\in\sigma_\alpha\cap I\>\le G(\Phi,R)\tp
$$

\begin{thm}\label{Jacobson} Let $\sigma$ be a net of ideals. Let $J$ be the Jacobson radical of the ring $R$. Then
	$$
	S(\sigma)\cap G(\Phi,R,J)\le T(\Phi,R,J)E(\sigma\cap J)\tp
	$$
\end{thm}

We denote by $N(\Phi,-)\le G(\Phi,-)$ the scheme normaliser of the group subscheme $T(\Phi,-)$. There is a natural homomorphism from the scheme $N(\Phi,-)$ to the Weyl group $W(\Phi)$ viewed as a constant scheme. This homomorphism is surjective on points; and its kernel is the subscheme $T(\Phi,R)$.

We denote by $\ovl{W}(\Phi)$ the image of the group $N(\Phi,\Z)$ under homomorphism $G(\Phi,\Z)\to G(\Phi,R)$ induced by the unique ring homomorphism $\Z\to R$. We have dropped the letter $R$ in this notation because this group almost independent of the ring: the group $\ovl{W}(\Phi)$ is isomorphic to the group $N(\Phi,\Z)$ if $2\ne 0$ on $R$; otherwise, it is isomorphic to the group $W(\Phi)$. 

Further let $\ovl{W}(\Phi,\sigma)\le \ovl{W}(\Phi)$ be the preimage of the group 
$$
W(\Phi,\sigma)=\{w\colon \forall\alpha\in\Phi \; \sigma_{w\alpha}=\sigma_\alpha\}\le W(\Phi)\tp
$$
It is easy to see that $\ovl{W}(\Phi,\sigma)=\ovl{W}(\Phi)\cap S(\sigma)$.

\begin{thm}\label{local} Let $R$ be a local ring; and let $\sigma$ be a net of ideals in it. Then
	$$
	S(\sigma)=T(\Phi,R)E(\sigma)\ovl{W}(\Phi,\sigma)\tp
	$$
\end{thm}

\begin{thm}\label{normal} Let $\sigma$ be the net of ideals in the ring $R$. Then $\hat{E}(\sigma)$ is a normal subgroup of $S(\sigma)$.
\end{thm}

\medskip

By $\Spec(R)$ we denote the set of prime ideals of the ring $R$. Consider the following subgroup in $S(\sigma)$
\begin{align*} 
	G(\sigma)
	\!=\!\big\{g\!\in\! G(\Phi,R)\,:\, \forall  \fr{P}\!\in\!\Spec\,(R)\;\; F_{\fr{P}}(g)\!\in\! T(\Phi,R_{\fr{P}})E(F_{\fr{P}}(\sigma))\!\le\! G(\Phi,R_{\fr{P}})\big\}\tp
\end{align*} 
It is clear that $\hat{E}(\sigma)\le G(\sigma)\le S(\sigma)$; the second inclusion here follows from Lemma~\ref{LocalisationForS}.

\begin{rem*} It can be shown that the group $G(\sigma)$ coincides with the net subgroup defined in~\cite{VavPlotII}. However, the definition above is more convenient for our purposes.
\end{rem*}

Let ${\fr{P}}\in\Spec(R)$. Set $\Delta'_{\fr{P}}=\{\alpha\in\Phi\colon \sigma_\alpha\not\sub \fr{P}\}$. It is easy to see that this is a closed subset of roots (i.e., if $\alpha$,$\beta\in\Delta'_{\fr{P}}$ and $\alpha+\beta\in\Phi$, then $\alpha+\beta\in\Delta'_{\fr{P}}$). Then $\Delta'_{\fr{P}}\cap -\Delta'_{\fr{P}}$ is a subsystem and it is invariant under the action of the group $W(\Phi,F_{\fr{P}}(\sigma))$. Therefore, $W(\Delta'_{\fr{P}}\cap -\Delta'_{\fr{P}})$ is a normal subgroup of $W(\Phi,F_{\fr{P}}(\sigma))$.

A prime ideal ${\fr{P}}\in\Spec(R)$ is called {\it defining}, if it is minimal by inclusion in the set  $\{\tilde{\fr{P}}\in\Spec(R)\colon \Delta'_{\tilde{\fr{P}}}=\Delta'_{\fr{P}}\}$. Let $\Theta\sub\Spec(R)$ be the set of all defining prime ideals.

\begin{thm}\label{finiteindex} In the previous notation $G(\sigma)$ is a normal subgroup of $S(\sigma)$ and there exists an injective homomorphism
	$$
	S(\sigma)/G(\sigma)\rightarrowtail \prod_{\fr{P}\in\Theta} W(\Phi,F_{\fr{P}}(\sigma))/W(\Delta'_{\fr{P}}\cap -\Delta'_{\fr{P}})\tp
	$$
\end{thm}

\begin{rems} \phan 
	\begin{enumerate}
		\item It is easy to see that if the ring $R$ is Noetherian, then the set $\Theta$ is finite.
		
		\item It can be shown that the group $S(\sigma)$ coincides with the normaliser of the group $G(\sigma)$.
	\end{enumerate}
\end{rems}

\begin{thm}\label{nilpotent} In the previous notation, if the subsystem $\Delta$ has no irreducible components of type $A_1$, and the ring $R$ has finite Bass--Serre dimension, then the quotient $G(\sigma)/\hat{E}(\sigma)$ is nilpotent-by-abelian.
\end{thm}

\medskip

Note that Theorems~\ref{Jacobson} and~\ref{local} for general linear group follow from~\cite{BV78}. For other classical group assuming 2 is invertible they follows from~\cite{VavMIAN}.

\section{The proof of Theorem \ref{Jacobson}}

By $J$ we always denote the Jacobson radical of the ring $R$. In this section we fix some system of positive roots $\Phi^+$ (then set $\Phi^-=-\Phi^+$) and some topological sort $(\,\le\,)$ of the corresponding Bruhat order on the root system $\Phi$. 

\begin{lem} \label{Jorder} Let $\sigma$ be a net of ideals. Let $g\in T(\Phi,R,J)E(\sigma\cap J)$. Then the element $g$ can be written as
	$$
	g=h\prod_{\gamma\in \Phi}x_{\gamma}(\xi_{\gamma}),
	$$ 
	where $\xi_\gamma \in \sigma_\gamma\cap J,$ $h\in T(\Phi,R,J),$ and the order of factors in the product coincides with the order $(\,\le\,)$.
\end{lem}

\begin{proof}
	The following elements will be called the generators of the group $T(\Phi,R,J) E(\sigma\cap J)$: the elements $x_{\gamma}(\xi)$, where $\xi\in \sigma_\gamma\cap J$, and the elements of the group $T(\Phi,R,J)$. The element $g$ can be expressed by a word in these generators. Since $\sigma$ is a net of ideals, it follows that after applying the following transformations of the word (which do not change the element expressed by this word) we again obtain a word in these generators. 
	\begin{enumerate}
		\item For $h\in T(\Phi,R,J)$ and $\xi\in(\sigma_\gamma\cap J)$ replace $hx_\gamma(\xi)$ by $x_{\gamma}(\xi')h$ or vice versa, where $\xi'$ is $\xi$ multiplied by suitable invertible element of the ring.
		
		\item Replace two adjacented elements of $T(\Phi,R,J)$ by their product.
		
		\item Replace $x_{\gamma}(\xi_1)x_\gamma(\xi_2)$ by $x_{\gamma}(\xi_1+\xi_2)$.
		
		\item For $\gamma_1\ne \pm\gamma_2$ swap $x_{\gamma_1}(\xi_1)$ and $x_{\gamma_2}(\xi_2)$ adding factors from the Chevalley commutator formula.
		
		\item Replace $x_{\gamma}(\xi_1)x_{-\gamma}(\xi_2)$ by $x_{-\gamma}(\xi_2(1+\xi_1\xi_2)^{-1})x_{\gamma}(\xi_1(1+\xi_1\xi_2))h_{\gamma}(1+\xi_1\xi_2)$. Here the element $(1+\xi_1\xi_2)$ is invertible because $\xi_1\in J$. The corresponding relation in the group $G(\Phi,R)$ follows from the relation in $\SL(2,R)$:
		\begin{align*} 
			&\begin{pmatrix}
				1 & \xi_1\\ 0 & 1
			\end{pmatrix}
			\begin{pmatrix}
				1 & 0 \\ \xi_2 & 1
			\end{pmatrix}
			\\
			&=\begin{pmatrix}
				1 & 0 \\ \xi_2(1+\xi_1\xi_2)^{-1} & 1
			\end{pmatrix}
			\begin{pmatrix}
				1 & \xi_1(1+\xi_1\xi_2)\\ 0 & 1
			\end{pmatrix}
			\begin{pmatrix}
				(1+\xi_1\xi_2) & 0 \\ 0 & (1+\xi_1\xi_2)^{-1}
			\end{pmatrix}\tp
		\end{align*} 
	\end{enumerate}
	
	It is also easy to see that using these five types of transformations, any word can be transformed to the required shape. 
\end{proof}

\begin{rem*} Alternatively, the above lemma can be proved by citing Theorem~2 in~\cite{VavPlotII} in order to separate positive root generators from negative root generators. After that it remans to apply transformations (1)--(4) from the proof above in order to place roots in the required order.
\end{rem*}
\begin{lem}\label{Jnilpotent} Let $\sigma$ be a net of ideals. Let $I\unlhd R$ be such ideal that $I^k=0$ for some $k\in\N$. Then $S(\sigma)\cap G(\Phi,R,I)\le T(\Phi,R,I)E(\sigma\cap I)$.
\end{lem}
\begin{proof}
	The proof is by induction on $k$.
	
	The base of induction is for $k=2$. Let $g\in S(\sigma)\cap G(\Phi,R,I)$. Since $g\in G(\Phi,R,I)$, it can be written as
	$$
	g=h\prod_{\gamma\in \Phi}x_{\gamma}(\xi_{\gamma}),
	$$ 
	where $h\in T(\Phi,R,I)$, $\xi_\gamma\in I$, and the order of factors in the product coincides with the order $(\,\le\,)$. It follows, for example, from~\cite[Proposition 2.3]{AbeSuzuki}. Now let us show that $\xi_{\gamma}\in \sigma_{\gamma}$ for any $\gamma\in \Phi\sm\Delta$. 
	
	By $(*)$ there exists $\alpha\in\Delta$ such that $\alpha+\gamma\in \Phi$ and $c_{\alpha,\gamma}\in R^*$. Since $g\in S(\sigma)$, it follows that $(\leftact{g}{e_\alpha})^{\alpha+\gamma}\in \sigma_{\alpha+\gamma}=\sigma_{\gamma}$ (Lemma \ref{EqualIdeals}). On the other hand, since $I^2=0$, it follows that $(\leftact{g}{e_\alpha})^{\alpha+\gamma}=\eps c_{\alpha,\gamma}\xi_{\gamma}$, where $\eps\in R^*$ depends on $h$. Therefore, $\xi_{\gamma}\in \sigma_{\gamma}$.
	
	Now let us perform the induction step from $k-1$ to $k$. It is clear that $\rho_{I^{k-1}}(S(\sigma))\le S(\rho_{I^{k-1}}(\sigma))\le G(\Phi,R/I^{k-1})$, and that $\rho_{I^{k-1}}(G(\Phi,R,I))\le G(\Phi,R/I^{k-1},\rho_{I^{k-1}}(I))$. Therefore, by induction hypothesis we have $\rho_{I^{k-1}}(S(\sigma)\cap G(\Phi,R,I))\le T(\Phi,R/I^{k-1})E(\rho_{I^{k-1}}(\sigma)\cap \rho_{I^{k-1}}(I))$.
	
	Since $E(\sigma\cap I)$ maps onto $E(\rho_{I^{k-1}}(\sigma)\cap \rho_{I^{k-1}}(I))$ surjectively (because $\rho_{I^{k-1}}(\sigma)\cap \rho_{I^{k-1}}(I)=\rho_{I^{k-1}}(\sigma\cap I)$ ), and (since the ideal $I$ is nilpotent) $T(\Phi,R,I)$ maps surjectively onto $T(\Phi,R/I^{k-1}$, $\rho_{I^{k-1}}(I))$, it follows that
	$$
	S(\sigma)\cap G(\Phi,R,I)\le T(\Phi,R,I)E(\sigma\cap I)G(\Phi,R,I^{k-1})\tp
	$$
	Therefore, using the base of induction and the fact that $(I^{k-1})^2=0$, we obtain
	\begin{align*}
		&S(\sigma)\cap G(\Phi,R,I)\le (T(\Phi,R,I)E(\sigma\cap I)G(\Phi,R,I^{k-1}))\cap S(\sigma)
		\\
		&= T(\Phi,R,I)E(\sigma\cap I)(G(\Phi,R,I^{k-1})\cap S(\sigma))=T(\Phi,R,I)E(\sigma\cap I)\tp
		\qedhere
	\end{align*}
\end{proof}

\begin{lem}\label{JNoeter} Let $\sigma$ be a net of ideals in a Noetherian ring $R$. Let $g\in S(\sigma)\cap G(\Phi,R,J),$ and let 
	$$
	g=h\prod_{\gamma\in \Phi}x_{\gamma}(\xi_{\gamma}),
	$$ 
	where $\xi_\gamma \in J$, $h\in T(\Phi,R,J),$ and the order of factors in the product coincides with the order $(\,\le\,)$. Then $\xi_{\gamma}\in \sigma_{\gamma}$.
\end{lem}

\begin{proof}
	Since $R$ is Noetherian, it follows (from the Krull In\-ter\-sec\-tion Theorem applied to localisations of the ring $R/\sigma_{\gamma}$ in all its maximal ideals) that $\sigma_{\gamma}=\bigcap_{k\in\N}(\sigma_{\gamma}+J^k)$. Therefore, it is enough to show that $\rho_{J^k}(\xi_\gamma)\in\rho_{J^k}(\sigma_{\gamma})$ for any $k\in\N$. 
	
	By Lemma~\ref{Jnilpotent} we have $\rho_{J^k}(g)\in T(\Phi,R/J^k, \rho_{J^k}(J))E(\rho_{J^k}(\sigma)\cap \rho_{J^k}(J))$; and by Lemma~\ref{Jorder} we have
	$$
	\rho_{J^k}(g)=h'\prod_{\gamma\in \Phi}x_{\gamma}(\xi'_{\gamma}),
	$$   
	where $\xi'_\gamma \in \rho_{J^k}(\sigma)\cap \rho_{J^k}(J)$ and $h'\in T(\Phi,R/J^k,\rho_{J^k}(J))$, and the order of factors in the product coincides with the order $(\,\le\,)$. 
	
	For an element of Chevalley group, if such an expression exists (i.e., it the element belong to the main cell of Gauss decomposition, see, for example, \cite{StepUniloc}), then it is unique, i.e. $\rho_{J^k}(\xi_\gamma)=\xi'_{\gamma}\in\rho_{J^k}(\sigma_{\gamma})$, q.e.d.
\end{proof}

\begin{rem*} The two previous lemmas can be deduced from the proof of Theo\-rem~2 in ~\cite{VavPlotII}. However, as already mentioned, in the present paper we use different (in fact equivalent) definitions.
\end{rem*}

Now we prove Theorem~\ref{Jacobson}. Let $g\in S(\sigma)\cap G(\Phi,R,J)$. Since $g\in G(\Phi,R,J)$, it follows that it can be written as
$$
g=h\prod_{\gamma\in \Phi}x_{\gamma}(\xi_{\gamma}),
$$ 
where $\xi_\gamma \in J$, $h\in T(\Phi,R,J)$, and the order of factors in the product coincides with the order $(\,\le\,)$. Let us prove that $\xi_{\gamma}\in \sigma_{\gamma}$ for any $\gamma\in\Phi$.

Let $R_1$ be such a finitely generated subring in $R$ that the element $h$ is defined in $R_1$, and $\xi_{\gamma}\in R_1$ for all $\gamma\in \Phi$. Let $R_2$ be the subring in $R$, obtained from $R_1$ by adding elements that are inverse to the elements from $R_1\cap R^*$. Then the ring $R_2$ is Noetherian because it is a localisation of $R_1$; and its Jacobson radical contains the ideal $R_2\cap J$. Also $R_2$ contains $1\over 2$ and/or~$1\over 3$ if it is required by  $(*)$. Therefore, it follows from Lemma~\ref{JNoeter} that $\xi_{\gamma}\in\sigma_\gamma\cap R_2\sub \sigma_\gamma$.

\section{The proof of Theorem \ref{local}}

Let $\Delta'\le\Phi$ be a closed set of roots (i.e if $\alpha,\beta\in\Delta'$ and $\alpha+\beta\in\Phi$, then $\alpha+\beta\in\Delta'$), and let $\Delta\le \Delta'$. For any ring $R$ consider the net of ideals $\sigma_{\Delta'}$ defined as follows
$$
(\sigma_{\Delta'})_\gamma=\begin{cases}
	R,&\gamma\in\Delta',\\
	(0),&\gamma\notin\Delta'.
\end{cases}
$$ 
Further set $E(\Delta',R)=E(\sigma_{\Delta'})$. Then the Zariski sheafification of the presheaf $T(\Phi,-)E(\Delta',-)$ is a semidirect product of the extended Chevalley group that corresponds to the subsystem $\Delta'\cap(-\Delta')$ and the unipotent subgroup that corresponds to set $\Delta'\sm(-\Delta')$. We denote this group by $GG(\Delta',R)$. 

Here by extended Chevalley group we mean the group
$$
GG(\Delta'\cap(-\Delta'),R)=T(\Phi,R)G(\Delta'\cap(-\Delta'),R).
$$

Note that there exists a closed subtorus $T_0(R)\le T(\Phi,R)$ (possibly trivial), such that
$
T(\Phi,R)=T(\Delta'\cap(-\Delta'),R)\times T_0(R),
$
where the factor $T(\Delta'\cap(-\Delta'),R)$ is a toric subgroup of the subsystem subgroup $G(\Delta'\cap(-\Delta'),R)\le G(\Phi,R)$. This follows from the fact that any surjective homomorphism of lattices splits, and from the duality between tori and lattices. Therefore, we have $GG(\Delta'\cap(-\Delta'),R)=G(\Delta'\cap(-\Delta'),R)\leftthreetimes T_0(R)$. In particular, this  is the group of points of a smooth affine group scheme, and the same is true for the group $GG(\Delta',R)$, because its second factor is isomorphic as a scheme to an affine space. 

We also set $\tilde{G}(\Delta',R)=S(\sigma_{\Delta'})$. It is clear that $GG(\Delta',R)\le \tilde{G}(\Delta',R)$.

For simplicity we will write $\ovl{W}(\Phi,\Delta')=\ovl{W}(\Phi,\sigma_{\Delta'})$.

The two following lemmas basically were proved in the author's paper~\cite{Gvoz2A1} (Lemmas 13 and 14). However, in order to adapt them to our assumptions the following changes must been made to the proofs: in the proof of the first lemma instead of the reference to Lemma~12, which uses stronger assumption on the subsystem~$\Delta$, one must refer to Theorem~\ref{Jacobson} of the present paper. The second Lemma can be deduded from the first one in the same way as in~\cite{Gvoz2A1}.

\begin{lem}\label{ConnectedComponentOverField} Let $R=L$ be an algebraically closed field. Then the subgroups $GG(\Delta',L)$ and $ \tilde{G}(\Delta',L)$ are closed in $G(\Phi,L),$ and the subgroup $GG(\Delta',L)$ is a connected component of identity of the group  $\tilde{G}(\Delta',L)$.
\end{lem}

\begin{lem}\label{DecomposeForField} Let $R=K$ be a field. Then $\tilde{G}(\Delta',K)=GG(\Delta',K)\ovl{W}(\Phi,\Delta')$.
\end{lem}

\medskip

We need one more lemma.

\begin{lem} \label{Weyl}\!\! Let $R\!$ be a local ring with nilpotent maximal ideal ${\M\unlhd R}$. Let $\sigma$ be a net of ideals in $R$. Set $\Delta'=\{\alpha\in\Phi\colon \sigma_{\alpha}=R\}$ \textup(this is such a closed subset of roots that $\rho_{\M}(\sigma)=\sigma_{\Delta'}$\textup). Let ${g\in S(\sigma),}$ and let $\rho_{\M}(g)=\ovl{g_1}\ovl{w},$ where $\ovl{g_1}\in GG(\Delta',R/\M),$ $\ovl{w}\in \ovl{W}(\Phi,\Delta')\le G(\Phi,R/\M)$. Then $\ovl{w}\in \ovl{W}(\Phi,\sigma)$ in a sense that the image of the element $\ovl{w}$ in the group $W(\Phi)$ preserves the net $\sigma$.
\end{lem}
\begin{proof}
	Since
	$$
	GG(\Delta',R/\M)=T(\Phi,R/\M)E(\Delta',R/\M),
	$$
	and $T(\Phi,R)E(\Delta',R)$ maps surjectively onto $T(\Phi,R/\M)E(\Delta',R/\M)$, it follows that multiplying $g$ by a suitable element of the group $T(\Phi,R)E(\Delta',R)\le S(\sigma)$, we may assume that $\ovl{g_1}=e$.
	
	Let $w\in \ovl{W}(\Phi)\le G(\Phi,R)$ be a preimage of the element $\ovl{w}$, and let $g=g_1w$. Then $g_1\in G(\Phi,R,\M)$. We must show that $w$ preserves the net $\sigma$ (acting on $\Phi$ by its image in $W(\Phi)$). 
	
	Let $\alpha\in\Phi$. Since $w$ has finite order, it is enough to show that $\sigma_{\alpha}\sub \sigma_{w\alpha}$. Let $\xi\in\sigma_{\alpha}$. Since $g\in S(\sigma)$, it follows that $(\leftact{g}{\xi e_{\alpha}})^{w\alpha}\in\sigma_{w\alpha}$. on the other hand, we have $(\leftact{g}{\xi e_{\alpha}})^{w\alpha}=\pm(\leftact{g_1}{\xi e_{w\alpha}})^{w\alpha}=\xi(\pm 1+\zeta)$, where $\zeta\in \M$. The element $(\pm 1+\zeta)$ is invertible; hence we have $\xi\in\sigma_{w\alpha}$, q.e.d.
\end{proof}

Now we prove Theorem~\ref{local}. Let $\M\unlhd R$ be the maximal ideal. Let $\Delta'$ be as in Lemma~\ref{Weyl}. Let $g\in S(\sigma)$. By Lemma~\ref{DecomposeForField}, $\rho_{\M}(g)=\ovl{g_1}\ovl{w}$, where $\ovl{g_1}\in GG(\Delta',R/\M)$, $\ovl{w}\in \ovl{W}(\Phi,\Delta')\le G(\Phi,R/\M)$. In addition, by Lemma~\ref{Weyl} we have ${\ovl{w}\in \ovl{W}(\Phi,\sigma)\le G(\Phi,R/\M)}$; hence multiplying $g$ by the corresponding element of the group $\ovl{W}(\Phi,\sigma)\le G(\Phi,R)$, we may assume that $\ovl{w}=e$. Further as in the proof of Lemma~\ref{Weyl}, multiplying $g$ by a suitable element of the group $T(\Phi,R)E(\Delta',R)\le S(\sigma)$, we may assume that $\ovl{g_1}=e$. Therefore, $g\in G(\Phi,R,\M)$. Then by Theorem~\ref{Jacobson} we have $g\in T(\Phi,R)E(\sigma)$. 

\section{Relative net subgroups}

Let $\sigma$ be a net of ideals in the ring $R$, and let $I\unlhd R$ be an ideal. Set
$$
S(\sigma,I)=S(\sigma)\cap G(\Phi,R,I)\tp
$$

Before we define relative elementary net subgroups, let us prove two technical lemmas.

\begin{lem}\label{rank2} Let $\alpha$,$\gamma\in\Phi$ and $\alpha\ne\pm\gamma$. Set
	$$
	U=\big\<x_{i\alpha+j\gamma}(t^j\theta)\colon i\alpha+j\gamma\in\Phi,\,i,j\in\Z,\,j>0,\, \theta\in R\big\>\le G(\Phi,R[t]),
	$$
	where $t$ is a free variable.
	Then the subgroup $U$ is normalised by the subgroup $G(\{\alpha,-\alpha\},R)$.
\end{lem}
\begin{proof}
	Let $g\in G(\{\alpha,-\alpha\},R)$. Choose a faithfully flat extension $\tilde{R}$ of the ring $R$ such that $g\in E(\{\alpha,-\alpha\},\tilde{R})$. It follows from the Chevalley commutator formula that
	$$
	g_1=\leftact{g}{x_{i_0\alpha+j_0\gamma}(t^{j_0}\theta_0)}=\prod_{i\alpha+j\gamma\in\Phi,\,j>0} x_{i\alpha+j\gamma}(t^j\theta_{i,j}),
	$$
	where $\theta_{i,j}\in \tilde{R}$. In addition, we have $g_1\in G(\Phi,R[t])$; hence in fact we have $\theta_{i,j}\in R$. 
\end{proof}

\begin{lem} \label{notA1} Let $\alpha\in \Delta$. Let $\Delta_0$ be the irreducible component of the system $\Delta$ containing the root $\alpha$. Assume that $\Delta_0\ne A_1$. Finally, let $\xi,\zeta,\eta\in R$. Then $t_{\alpha}^{\zeta,\eta}(\xi)\in E(\Delta_0,R)$.
\end{lem}
\begin{proof}
	Since $\Delta_0\ne A_1$, it follows that the root $\alpha$ is contained in a subsystem of type $A_2$, $C_2$ or $G_2$. In the first case the statement follows from Lemma~3 in~\cite{StepVavDecomp}. In the second case, from Lemmas~1.2 (if $\alpha$ is a long root) and~1.6 (if it is short) in~\cite{KopejkoSpStab}. The third case can not arise because we assume that $\Delta$ is a proper subsystem in an irreducible root system~$\Phi$.
\end{proof}

\begin{rem*} For the system $G_2$ this statement holds true as well; however, this is inessential for our purposes.
\end{rem*}

Let $\sigma$ be a net of ideals, and let $I\unlhd R$ be an ideal. Set
$$
\hat{E}(\sigma,I)=\big\<\big\{x_\alpha(\xi)\colon \alpha\in\Phi,\, \xi\in I\cap\sigma_\alpha\big\}\cup\big\{t_{\alpha}^{\zeta,\eta}(\xi)\colon \alpha\in\Delta,\;\xi\in I,\,\zeta,\eta\in R\big\}\big\>^{\hat{E}(\sigma)}\tp
$$

\begin{prop}\label{relativegenerators} Let $\sigma$ be a net of ideals, and let $I\unlhd R$ be an ideal. Then the following equality holds true\textup:
	\begin{align*} 
		\hat{E}(\sigma,I)&=\big\<\{\leftact{x_{-\alpha}(\zeta)}x_\alpha(\xi)\,:\, \alpha\in\Phi,\, \xi\in I\cap\sigma_\alpha,\,\zeta\in \sigma_{-\alpha}\big\}
		\\
		&\cup\{t_{\alpha}^{\zeta,\eta}(\xi)\,:\, \alpha\in\Delta,\;\xi\in I,\,\zeta,\eta\in R\}\>\tp
	\end{align*}
\end{prop}
\begin{proof}
	We denote by $H$ the right hand side of the equality in question. Clearly, $H\le \hat{E}(\sigma,I)$. In order to prove the converse inclusion, it is enough to prove that $H$ is a normal subgroup of $\hat{E}(\sigma)$, i.e. that generators of $H$ after conjugation by generators of $\hat{E}(\sigma)$ remain in $H$.
	
	\begin{enumerate}
		
		\item $\leftact{x_{\gamma}(t)}{t_{\alpha}^{\zeta,\eta}(\xi)}\in H$, where $t\in \sigma_{\gamma}$, $\alpha\in\Delta$, $\xi\in I$, $\zeta,\eta\in R$.
		\begin{enumerate}
			\item $\gamma=\pm\alpha$. Follows from Lemma~\ref{ConjugeteTransvection}.
			
			\item $\gamma\ne\pm\alpha$. If $i\alpha+j\gamma\in\Phi$ and $j>0$, then $t^j\in \sigma_{i\alpha+j\gamma}$. Hence it is enough to show that
			$$
			[x_\gamma(t),t_{\alpha}^{\zeta,\eta}(\xi)]=\prod_{i\alpha+j\gamma\in\Phi,\,j>0} x_{i\alpha+j\gamma}(t^j\xi\theta_{i,j}),
			$$ 
			where $\theta_{i,j}\in R$. In addition, it is enough to prove the equality above in case where $R=\Z[t,\xi,\zeta,\eta]$ is the ring of polynomials. It follows from Lemma~\ref{rank2} that 
			$$
			[x_\gamma(t),t_{\alpha}^{\zeta,\eta}(\xi)]=\prod_{i\alpha+j\gamma\in\Phi,\,j>0} x_{i\alpha+j\gamma}(t^j\tilde{\theta_{i,j}})\tp
			$$ 
			In addition, this commutator belongs to the subgroup $G(\Phi,R,(\xi))$; hence all the elements $t^j\tilde{\theta_{i,j}}$ (hence, in fact, the elements $\tilde{\theta_{i,j}}$) are divisible by $\xi$.
		\end{enumerate}
		
		\smallskip
		
		\item $\leftact{x_{\gamma}(t)x_{-\alpha}(\zeta)}{x_{\alpha}(\xi)}\in H$, where $t\in \sigma_{\gamma}$, $\zeta\in \sigma_{-\alpha}$, $\xi\in I\cap\sigma_{\alpha}$.
		
		\begin{enumerate}
			\item $\gamma=-\alpha$. This case is trivial.
			
			\item $\gamma\ne \pm\alpha$. Follows from the Chevalley commutator formula in the same way as in~\cite{VasersteinChevalley}.
			
			\item $\gamma=\alpha$. For $\alpha\in\Delta$ we obtain a special case of what we prove in the first item. Let $\alpha\notin\Delta$. By $(*)$ there exists a root $\beta\in\Delta$ such that $\alpha-\beta\in \Phi$, and the structure constant $c=c_{\beta,\alpha-\beta}=\pm c_{-\beta,\alpha}$ is invertible. We express $x_{\alpha}(\xi)$ from the Chevalley commutator formula for the commutator  $[x_{\alpha-\beta}(\xi),x_{\beta}(\pm c^{-1})]$ (by Lemma~\ref{EqualIdeals} we have $\sigma_{\alpha-\beta}=\sigma_{\alpha}$). Further we perform as in~\cite{VasersteinChevalley}. 
		\end{enumerate}
		
		\smallskip
		
		\item $\leftact{t_{\gamma}^{r,s}(t)x_{-\alpha}(\zeta)}{x_{\alpha}(\xi)}\in H$, where $\gamma\in\Delta$, $t,r,s\in R$, $\zeta\in \sigma_{-\alpha}$, $\xi\in I\cap\sigma_{\alpha}$.
		\begin{enumerate}
			\item $\gamma=\pm\alpha$. Follows from Lemma~\ref{ConjugeteTransvection}.
			
			\item $\gamma\ne\pm\alpha$. Then $\leftact{t_{\gamma}^{r,s}(t)x_{-\alpha}(\zeta)}{x_{\alpha}(\xi)}=\leftact{g}{h}$, where $g=\leftact{t_{\gamma}^{r,s}(t)}{x_{-\alpha}(\zeta)}$ and $h=\leftact{t_{\gamma}^{r,s}(t)}{x_{\alpha}(\xi)}$. In addition, by Lemma~\ref{rank2} we have
			\begin{align*}
				g=[t_{\gamma}^{r,s}(t),x_{-\alpha}(\zeta)]\cdot x_{-\alpha}(\zeta)=\prod_{-i\alpha+j\gamma\in\Phi,\,i>0} x_{-i\alpha+j\gamma}(\zeta^i\theta_{i,j})\cdot x_{-\alpha}(\zeta),
				\\
				h=[t_{\gamma}^{r,s}(t),x_{\xi}(\zeta)]\cdot x_{\alpha}(\xi)=\prod_{i\alpha+j\gamma\in\Phi,\,i>0} x_{i\alpha+j\gamma}(\xi^i\theta'_{i,j})\cdot x_{-\alpha}(\zeta)\in H.
			\end{align*}
			It remains to notice that the factors that we decompose the element $g$ into, normalise the subgroup $H$ because of what we prove pre\-vi\-ously.
		\end{enumerate}
		
		\smallskip
		
		\item $\leftact{t_{\gamma}^{r,s}(t)}{t_{\alpha}^{\zeta,\eta}(\xi)}\in H$, where $\alpha,\gamma\in\Delta$, $\xi\in I$, $r,s,t,\zeta,\eta\in R$.
		\begin{enumerate}
			\item $\gamma=\pm\alpha$. Follows from Lemma~\ref{ConjugeteTransvection}.
			
			\item $\gamma\ne\pm\alpha$. If the connected component of the subsystem $\Delta$ containing the root $\gamma$ is not of type $A_1$, then by Lemma~\ref{notA1} the element $t_{\gamma}^{r,s}(t)$ is contained in $E(\Delta,R)$; hence by what we prove pre\-vi\-ously it normalises the subgroup $H$. Otherwise, the roots $\alpha$ and $\gamma$ are in different components; hence $\leftact{t_{\gamma}^{r,s}(t)}{t_{\alpha}^{\zeta,\eta}(\xi)}=t_{\alpha}^{\zeta,\eta}(\xi)\in H$.
		\end{enumerate}
	\end{enumerate}
\end{proof}

\begin{prop}\label{withoutA1} In the previous notation\textup, assume that the subsystem $\Delta$ has no irreducible components of type $A_1$. Then
	\begin{align*} 
		\hat{E}(\sigma,I)&=\big\<\{x_\alpha(\xi)\,:\, \alpha\in\Phi,\, \xi\in I\cap\sigma_\alpha\}\big\>^{\hat{E}(\sigma)}
		\\
		&=\big\<\{\leftact{x_{-\alpha}(\zeta)}x_\alpha(\xi)\,:\, \alpha\in\Phi,\, \xi\in I\cap\sigma_\alpha,\,\zeta\in \sigma_{-\alpha}\}\big\>\tp
	\end{align*} 
	
	In this case we will denote this group just by $E(\sigma,I)$.
\end{prop}

\begin{proof}
	By Proposition~\ref{relativegenerators} it is enough to prove that an element $t_{\alpha}^{\zeta,\eta}(\xi)$, where $\alpha\in\Delta$, $\xi\in I$, $\zeta,\eta\in R$, can be expressed as a product of elements $\leftact{x_{-\alpha}(\zeta_i)}x_\alpha(\xi_i)$, where $\alpha\in\Delta$, $\xi_i\in I$, $\zeta_i\in R$.
	
	In addition, it is enough to prove that in case where $\xi$,$\zeta$ and $\eta$ are free variable in the polynomial ring $R=\Z[\xi,\eta,\zeta]$, and $I=(\xi)$. It is clear that $t_{\alpha}^{\zeta,\eta}(\xi)\in G(\Delta,R,I)$. Further since $\Delta$ has no irreducible components of type $A_1$, it follows by Lemma~\ref{notA1} that $t_{\alpha}^{\zeta,\eta}(\xi)\in E(\Delta,R)$. Note that $R=R_1\oplus I$, where $R_1$ is a subring in $R$, which implies that $G(\Delta,R,I)\cap E(\Delta,R)=E(\Delta,R,I)$ (see Lemma~3.2 in~\cite{StepUniloc}). Therefore, $t_{\alpha}^{\zeta,\eta}(\xi)\in E(\Delta,R,I)$, and by~\cite{VasersteinChevalley}, in can be expressed in the desired way.
\end{proof}

\section{The proof of Theorem \ref{normal}}

We adapt for our context the idea of the proof given in~\cite{Taddei} of the theorem on normality of the elementary subgroup in Chevalley group.

If $\sigma$ is a net of ideals in the ring $R$, then we denote by $\sigma[t]$ the corresponding net in the polynomial ring $R[t]$; i.e the ideal ${\sigma[t]_\alpha\unlhd R[t]}$ is generated by the set $\sigma_{\alpha}$. 

\begin{lem}\label{denominators} Let $\sigma$ be a net of ideals in the ring $R,$ let $s\in R$ be a non nilpotent element\textup, and let $h(t)\in \hat{E}(F_s(\sigma[t]),(t))\le G(\Phi,R_s[t],(t))$. Then for a big enough $N\in\N$ there exists $g(t)\in \hat{E}(\sigma[t],(t))\le G(\Phi,R[t])$ such that $F_s(g(t))=h(s^N t)$.
\end{lem}

\begin{proof}
	It is enough to prove the statement in question for the generators of the group $\hat{E}(F_s(\sigma[t]), (t))$ given in Proposition~\ref{relativegenerators}.
	
	\begin{enumerate}
		\item $h(t)=t_\alpha^{\zeta(t),\eta(t)}(t\xi(t))$, $\alpha\in\Delta, \zeta(t),\eta(t),\xi(t)\in R_s[t]$.
		
		\noindent Let the degree of $s$ on the denominators of $\zeta(t),\eta(t)$ and $\xi(t)$ be not greater than~$n$. Then we have
		\begin{align*} 
			h(s^{3n}t)&=t_\alpha^{\zeta(s^{3n}t),\eta(s^{3n}t)}(s^{3n}t\xi(s^{3n}t))
			\\
			&=t_\alpha^{s^n\zeta(s^{3n}t),s^n\eta(s^{3n}t)}(s^{n}t\xi(s^{3n}t))\in F_s(\hat{E}(\sigma[t],(t)))\tp
		\end{align*} 
		
		\smallskip
		
		\item $h(t)=\leftact{x_{-\alpha}(\zeta(t))}x_\alpha(t\xi(t))$, $\alpha\in\Phi$, $\xi(t)\in F_s(\sigma[t]_\alpha))$, $\zeta(t)\in F_s(\sigma[t]_{-\alpha})$. 
		
		Here we use that $(t)\cap F_s(\sigma[t]_\alpha)=(t)F_s(\sigma[t]_\alpha)$.
		
		We may assume that $\alpha\in\Phi\sm\Delta$ because otherwise such a generator would coincide with the one from the first case. 
		
		Then by $(*)$ there exists $\beta\in\Delta$ such that $\alpha-\beta=\gamma\in\Phi$ and the structure constant $c=c_{\alpha,-\beta}=\pm c_{\beta,\gamma}$ is invertible in $R$. 
		
		For simplicity we denote $t'=s^{2n}t$.
		
		We express $x_\alpha(t'\xi(t'))$ from the Chevalley commutator formula for the commutator $[x_\gamma(s^{n}t\xi(t')),x_\beta(\pm c^{-1}s^n)]$. If $n$ is big enough, then the remaining terms in the formula after conjugation by the element $x_{\!-\!\alpha}(\!\zeta\!(t')\!),$ will be in $F_s(\hat{E}(\sigma[t],(t)))$, which is easy to see again from the Chevalley commutator formula. Similarly, we obtain that
		$$
		\leftact{x_{-\alpha}(\zeta(t'))}x_\gamma(s^{n}t\xi(t'))\in F_s(\hat{E}(\sigma[t],(t))),
		$$
		and 
		$$
		\leftact{x_{-\alpha}(\zeta(t'))}x_\beta(\pm c^{-1}s^n)\in F_s(\hat{E}(\sigma[t])).
		$$
		Therefore,
		$$
		\leftact{x_{-\alpha}(\zeta(t'))}[x_\gamma(s^{n}t\xi(t')),x_\beta(\pm c^{-1}s^n)]\in F_s(\hat{E}(\sigma[t],(t))),
		$$
		and
		\begin{equation*} 
			\leftact{x_{-\alpha}(\zeta(t'))}x_\alpha(t'\xi(t'))\in F_s(\hat{E}(\sigma[t],(t))).
			\qedhere
		\end{equation*} 
	\end{enumerate}
\end{proof}

\begin{lem}\label{stothepowern} Let $\sigma$ be a net of ideals in the ring $R,$ let $s \in R$ be a non nilpotent element\textup, and let $g(t)\in G(\Phi,R[t],(t))$ be such that $F_s(g(t))\in \hat{E} (F_s(\sigma[t]),(t))$. Then we have $g(s^Nt)\!\in\! \hat{E}(\sigma[t],(t))$ for big enough $N\!\in\!\N$.
\end{lem}
\begin{proof}
	By Lemma~\ref{denominators} for big enough $n_1\in\N$ there exists $g_1(t)\in \hat{E}(\sigma[t],(t))$ such that $F_s(g(s^{n_1}t))=F_s(g_1(t))$. Then it is easy to see that for big enough $n_2$ we have $g(s^{n_1+n_2}t)=g_1(s^{n_2}t)\in \hat{E}(\sigma[t],(t))$.
\end{proof}

\begin{lem}\label{localnorm} Assume that the ring $R$ is local. Let $\sigma$ be a net of ideals in it. Then the subgroup $S(\sigma)\le G(\Phi,R)\le G(\Phi,R[t])$ normalises the subgroup
	$$
	\hat{E}(\sigma[t],(t))\le G(\Phi,R[t]).
	$$
\end{lem}

\begin{proof}
	By theorem~\ref{local}, $S(\sigma)=T(\Phi,R)E(\sigma)\ovl{W}(\Phi,\sigma)$. The second factor belongs to the subgroup $\hat{E}(\sigma[t],(t))$, and conjugation by the first and the third factors preserves the set of its generators.
\end{proof}

Now we prove Theorem~\ref{normal}. Let $h\in S(\sigma)$. Set 
$$
g(t)=\leftact{h}{x(t)}\in G(\Phi,R[t],(t)),
$$
where either $x(t)=x_\alpha(\xi t)$, where $\alpha\in\Phi$, $\xi\in\sigma_{\alpha}$, or $x(t)=t_\alpha^{\zeta,\eta}(\xi t)$, where $\alpha\in\Delta$, $\zeta,\eta,\xi\in R$. We must prove that $g(1)\in \hat{E}(\sigma)$, and to do so it is enough to prove that $g(t)\in \hat{E}(\sigma[t],(t))$.

Set $I=\{\theta\in R\colon g(\theta t)\in \hat{E}(\sigma[t],(t))\}$. This set is an ideal of the ring $R$. Indeed, $I+I\sub I$ because $g(t_1+t_2)=g(t_1)g(t_2)$; and $RI\sub I$ because the replacement of $t$ by $rt$, where $r\in R$, preserves the subgroup $\hat{E}(\sigma[t],(t))$. Therefore, we must prove that $1\in I$.

Assume the converse, then $I\sub\M$ for some maximal ideal $\M$. It follows from Lemma~\ref{localnorm} that $F_{\M}(g(t))\in \hat{E}(F_{\M}(\sigma[t]),(t))$. Then it is clear that $F_{s}(g(t))\in \hat{E}(F_{s}(\sigma[t]),(t))$ for some $s\in R\sm\M$. Therefore, by Lemma~\ref{stothepowern} we have $g(s^Nt)\in \hat{E}(\sigma[t],(t))$, i.e. $s^N\in I\sub\M$. But then we have $s\in\M$. This is a contradiction.

\section{The standard commutator formula}

In this Section we obtain a corollary from Theorem~\ref{normal}, which stands as an analog of the {\it standard commutator formula}. This is how the equality
$$
[G(\Phi,R,I), E(\Phi,R)]=E(\Phi,R,I)
$$
is called. This formula has many proofs. We adapt for our context the one given in~\cite{StepUniloc}.

\begin{cor}\label{StandComm} In the previous notation, assume that the subsystem $\Delta$ has no irreducible components of type $A_1$. Then the following inclusion holds true
	$$
	[S(\sigma,I),E(\sigma)]\le E(\sigma,I)\tp
	$$
\end{cor}

\begin{proof}
Let $g\in S(\sigma,I)$, and $h\in E(\sigma)$. Recall that given a net $\sigma$ we defined certain net $\dot{\sigma}$ in the ring $R[G]$. We identify the element $h$ with its image in the group $G(\Phi,R[G])$, then we have $h\in E(\dot{\sigma})$. By definition of the net $\dot{\sigma}$ we have $g_{\gen,R}\in S(\dot{\sigma})$. Hence by Theorem~\ref{normal}, using Proposition~\ref{withoutA1}, we obtain $[g_{\gen,R},h]\in E(\dot{\sigma})$; i.e. the following equality holds true
$$
[g_{\gen,R},h]=\prod_i x_{\beta_i}(\tilde{\xi}_i),\eqno{(\#)}
$$
where $\beta_i\in\Phi$, and $\tilde{\xi}_i\in\dot{\sigma}_{\beta_i}$.

It is easy to see that $R[G]=R\oplus I_{\aug,R}$ if we identify $R$ with its image in $R[G]$. In addition, we have $\dot{\sigma}_{\alpha}=\sigma_{\alpha}\oplus (\dot{\sigma}_{\alpha}\cap I_{\aug,R})$. Indeed, let $t=t_1+t_2$, where $t\in\dot{\sigma}_{\alpha}$, $t_1\in R$ and $t_2\in I_{\aug,R}$. Then by Lemma~\ref{Sgeneric}, using the fact that the identity element belongs to $S(\sigma)$, we obtain that $t_1=i(t)\in\sigma_{\alpha}$, where $i$ is the augmentation homomorphism. Thus if we consider $t_1$ as an element of the ring $R[G]$, then we have $t_1\in\dot{\sigma}_{\alpha}$; hence we have $t_2\in\dot{\sigma}_{\alpha}$.

Therefore, using additivity of $x_\alpha$, we may assume that each of the elements $\tilde{\xi}_i$ in the equality $(\#)$ belong either to $\sigma_{\beta_i}$, or to $\dot{\sigma}_{\beta_i}\cap I_{\aug,R}$. 

Note that if we remove from the right hand side of the equality $(\#)$ the factors that have $\tilde{\xi}_i\in \dot{\sigma}_{\beta_i}\cap I_{\aug,R}$, then we obtain the element that is equal to 
$$
i_*([g_{\gen,R},h])\in G(\Phi,R)\le G(\Phi,R[G]).
$$
On the one hand, this is the identity element because $i_*(g_{\gen,R})=e$.On the other hand, its image under projection $E(\dot{\sigma})\to E(\dot{\sigma})/E(\dot{\sigma},I_{\aug,R})$ coincides with the image of the element  $[g_{\gen,R},h]$. Therefore, we have
$$
[g_{\gen,R},h]\in E(\dot{\sigma},I_{\aug,R}),
$$
and by Proposition~\ref{withoutA1} the following equality holds
$$
[g_{\gen,R},h]=\prod_i \leftact{x_{-\gamma_i}(\zeta_i)}{x_{\gamma_i}(\xi_i)},
$$
where $\gamma_i\in\Phi$, $\xi_i\in\dot{\sigma}_{\gamma_i}\cap I_{\aug,R}$, and $\zeta_i\in\dot{\sigma}_{-\gamma_i}$.

Hence we obtain that
$$
[g,h]=\prod_i \leftact{x_{-\gamma_i}(\zeta_i(g))}{x_{\gamma_i}(\xi_i(g))}\tp
$$
By Lemma~\ref{Sgeneric}, using that $g\in S(\sigma,I)$, we obtain that $\xi_i(g)\in\sigma_{\gamma_i}\cap I$, and $\zeta_i(g)\in\sigma_{-\gamma_i}$; i.e. $[g,h]\in E(\sigma,I)$.
\end{proof}

\begin{rem*} It is easy to see that if the group $E(\Delta,R)$ is perfect, then the converse inclusion holds true as well because in this case we have
$$
[E(\sigma),E(\sigma,I)]=E(\sigma,I).
$$
\end{rem*}

\section{The proof of Theorem~\ref{finiteindex}}

We need the following lemma.

\begin{lem}\label{TEcapW} Let $\sigma$ be a net of ideals in the ring $R,$ and let $\fr{P}\in\Spec(R)$. Then the group $(T(\Phi,R_{\fr{P}})E(F_{\fr{P}}(\sigma)))\cap \ovl{W}(\Phi,F_{\fr{P}}(\sigma))$ coincides with the preimage of the subgroup $W(\Delta'_{\fr{P}}\cap -\Delta'_{\fr{P}})\le W(\Phi)$ in the group $\ovl{W}(\Phi)$.
\end{lem}
\begin{proof}
Clearly, the preimage of the subgroup $W(\Delta'_{\fr{P}}\cap -\Delta'_{\fr{P}})\le W(\Phi)$ in the group $\ovl{W}(\Phi)$ is contained in $(T(\Phi,R_{\fr{P}})E(F_{\fr{P}}(\sigma)))\cap \ovl{W}(\Phi,F_{\fr{P}}(\sigma))$. Now we prove the converse inclusion. Let 
$$
w\in(T(\Phi,R_{\fr{P}})E(F_{\fr{P}}(\sigma)))\cap \ovl{W}(\Phi,F_{\fr{P}}(\sigma)),
$$
then we have
$$
\rho_{\fr{P}}(w)\in GG(\Delta'_{\fr{P}},R/\fr{P})\cap \ovl{W}(\Phi)\tp
$$
Hence it is easy to see that the image of the element $\rho_{\fr{P}}(w)$, and hence of the element~$w$, in the group $W(\Phi)$ belongs to the subgroup $W(\Delta'_{\fr{P}}\cap -\Delta'_{\fr{P}})$. 
\end{proof}

Now we prove Theorem~~\ref{finiteindex}.

It follows from Theorem~\ref{local} and Lemma~\ref{TEcapW} that $T(\Phi,R_{\fr{P}})E(F_{\fr{P}}(\sigma))$ is a normal subgroup in $S(F_{\fr{P}}(\sigma))$, and that $$S(F_{\fr{P}}(\sigma))/(T(\Phi,R_{\fr{P}})E(F_{\fr{P}}(\sigma)))\simeq W(\Phi,F_{\fr{P}}(\sigma))/W(\Delta'_{\fr{P}}\cap -\Delta'_{\fr{P}})\tp$$

Therefore, localisation homomorphisms $\map{F_{\fr{P}}}{S(\sigma)}{S(F_{\fr{P}}(\sigma))}$ define ho\-mo\-mo\-rphism.
$$
\map{\psi}{S(\sigma)}{\prod_{\fr{P}\in\Theta} W(\Phi,F_{\fr{P}}(\sigma))/W(\Delta'_{\fr{P}}\cap -\Delta'_{\fr{P}})}\tp
$$

Set $H=\Ker\psi$. By definition $G(\sigma)\le H$; it remains to prove the converse inclusion.

Let $g\in H$ and $\fr{P}\in\Spec(R)$. By Theorem~\ref{local} $F_{\fr{P}}(g)=g_1w$, where $g_1\in T(\Phi,R_{\fr{P}})E(F_{\fr{P}}(\sigma))$ and $w\in \ovl{W}(\Phi,\sigma)$. We must prove that under projection on $W(\Phi)$ the element $w$ gets in the subgroup $W(\Delta'_{\fr{P}}\cap -\Delta'_{\fr{P}})$.

By Zorn lemma, there exists a prime ideal $\fr{P}_1\sub \fr{P}$, containing all the $\sigma_\alpha$, $\alpha\in\Phi\sm\Delta'_{\fr{P}}$, and minimal among prime ideals with such properties. Then $\Delta'_{\fr{P}_1}=\Delta'_{\fr{P}}$, and $\fr{P}\in\Theta$. Therefore, $F_{\fr{P}_1}(g)\in T(\Phi,R_{\fr{P}_1})E(F_{\fr{P}_1}(\sigma))$, but $F_{\fr{P}_1}=F_{F_{\fr{P}}(\fr P_1)}\circ F_{\fr{P}}$, which, using Lemma~\ref{TEcapW}, implies that the element $F_{F_{\fr{P}}(\fr{P}_1)}(w)\in\ovl{W}(\Phi,\sigma)$ under projection on  $W(\Phi)$ gets in the subgroup
$$
W(\Delta'_{\fr{P}_1}\cap -\Delta'_{\fr{P}_1})=W(\Delta'_{\fr{P}}\cap -\Delta'_{\fr{P}}).
$$
In addition, images of $w$ and $F_{F_{\fr{P}}(\fr{P}_1)}(w)$ under this projection coincide.

\section{The proof of Theorem~\ref{nilpotent}}

We adapt for our context the given in~\cite{HazVav} proof of the theorem on $K_1$ being nilpotent. In addition we simplify this proof using generic element.

Let $R$ be a commutative ring, and let $s\in R$. Following~\cite{BakNonabelian} and~\cite{HazVav}, we use the following notation
$$
\tilde{R}_s=\varinjlim_i\varprojlim_n R_i/(s^n),
$$ 
where the limit is taken over all finitely generated subrings $R_i$ of $R$, which contain $s$. Since the ring $R$ is the limit of its subrings $R_i$, we have a canonical homomorphism.
$$
\map{\tilde{F}_s}{R}{\tilde{R}_s}\tp
$$

If the ring $R$ is finitely generated, then $\tilde{R}_s$ coincides with the usual completion:
$$
\tilde{R}_s=\hat{R}_s=\varprojlim_n R/(s^n)\tp
$$

By $\delta(R)$ we denote Bass–Serre dimension of the ring $R$. We will use only the two following properties of this dimension (see~\cite{BakNonabelian} or~\cite{HazVav}).

\begin{lem}\label{InductionDim} \begin{enumerate}
	\item $\delta(R)=0$ if and only if the ring $R$ is semilocal.
	
	\item Let $\delta(R)<\infty,$ then there exists a finite collection of maximal ideals $\M_1,$ $\dots,$ $\M_k$ of the ring $R$ such that for any $s\in R\sm\left(\bigcup_k \M_k \right)$ we have $\delta(\tilde{R}_s)<\delta(R)$. 
\end{enumerate}
\end{lem}

Let us prove several technical lemmas.

\begin{lem}\label{functor} Let $\sigma_1,$ $\sigma_2$ be nets of ideals in rings $R_1,$ $R_2$ correspondently, and homomorphism $\map{f}{R_1}{R_2}$ be such that $f(\sigma_1)\sub \sigma_2$. Then 
$$
{f(G(\sigma_1))\le G(\sigma_2)}.
$$
\end{lem}

\begin{proof}
This follows from the fact that a preimage of a prime ideal is a prime ideal.
\end{proof}

\begin{lem}\label{GForSemiLocal} Let $\sigma$ be a net of ideals in a semilocal ring $R$. Then the following equality holds true\textup:
$$
G(\sigma)=T(\Phi,R)E(\sigma)\tp
$$
\end{lem}
\begin{proof}
The right hand side is obviously contained in the left hand side. Let us prove the converse inclusion.

Let $J$ be the Jacobson radical of the ring $R$. Then the quotient ring $R/J$ is a direct product of a finite collection of fields. Hence be Lemma~\ref{functor} and by definition of the group $G(\sigma)$ we have
$$
\rho_J(G(\sigma))\le G(\rho_J(\sigma))\le T(\Phi,R/J)E(\rho_J(\sigma))\tp
$$
Since the group $T(\Phi,R)E(\sigma)$ maps onto the group $T(\Phi,R/J)E(\rho_J(\sigma))$ sur\-jec\-tively, it follows that $G(\sigma)\le G(\Phi,R,J)T(\Phi,R)E(\sigma)$. Therefore, applying Theorem~\ref{Jacobson}, we obtain
\begin{align*} 
	G(\sigma)&\le (G(\Phi,R,J)T(\Phi,R)E(\sigma))\cap S(\sigma)
	\\
	&=(G(\Phi,R,J)\cap S(\sigma))T(\Phi,R)E(\sigma)=T(\Phi,R)E(\sigma)\tp
	\qedhere
\end{align*} 
\end{proof}

\begin{rem*} The previous lemma can also be deduced from Theorem~1 in~\cite{VavPlotII}.
\end{rem*}
\begin{lem}\label{Denominators2} Assume that the subsystem $\Delta$ has no irreducible components of type $A_1$. Let $\sigma$ be a net of ideals in a Noetherian ring $R,$ let $s\in R$ be a non nilpotent element\textup, and let $h\in F_s^{-1}(E(F_s(\sigma)))\le G(\Phi,R)$. Then there exists $N\in\N$ such that the following inclusion holds true
$$
[S(\sigma,(s^N)),h]\sub E(\sigma)\tp
$$
\end{lem}

\begin{proof}
Recall that given a net $\sigma$ we defined certain net $\dot{\sigma}$ in the ring $R[G]$. We identify the element $h$ with its image in the group $G(\Phi,R[G])$, then $F_s(h)\in E(\dot{\sigma})$. By definition of $\dot{\sigma}$ we have $g_{\gen,R}\in S(\dot{\sigma},I_{\aug,R})$. Then $F_s(g_{\gen,R})\in S(F_s(\dot{\sigma}),F_s(I_{\aug,R}))$, and by Corollary~\ref{StandComm} we have
$$
F_s([g_{\gen,R},h])\in E(F_s(\dot{\sigma}),F_s(I_{\aug,R})).
$$
Therefore, by Proposition~\ref{withoutA1} we have
$$
F_s([g_{\gen,R},h])=\prod_i \leftact{x_{-\gamma_i}\left(\tfrac{\zeta_i}{s^{N_1}}\right)}{x_{\gamma_i}\left(\tfrac{\xi_i}{s^{N_1}}\right)}\in G(\Phi,R_s[G]),
$$
where $\gamma_i\in\Phi$, $\xi_i\in\dot{\sigma}_{\gamma_i}\cap I_{\aug,R}$, $\zeta_i\in\dot{\sigma}_{-\gamma_i}$, and $N_1$ is a sufficiently large natural number.

Hence we obtain that for $g\in S(\sigma)$ the following equality holds true
$$
F_s([g,h])=\prod_i \leftact{x_{-\gamma_i}\left(\tfrac{\zeta_i(g)}{s^{N_1}}\right)}{x_{\gamma_i}\left(\tfrac{\xi_i(g)}{s^{N_1}}\right)}\in G(\Phi,R_s)\tp \eqno{(\#)}
$$

Since the ring $R$ is Noetherian, it follows that for a sufficiently large $N_2$ the homomorphism $\map{F_s}{G(\Phi,R,(s^{N_2}))}{G(\Phi,R_s)}$ is injective (see~\cite{BakNonabelian} or~\cite{HazVav}).

Set $N_3=3N_1+N_2$. Then for any $g\in S(\sigma,(s^{N_2}))$ such that $\xi_i(g)\in s^{N_3}\sigma_{\gamma_i}$ for all $i$ we have $[g,h]\in E(\sigma)$. Indeed, let $\xi_i(g)=s^{N_3}\theta_i$, where $\theta_i\in\sigma_{\gamma_i}$. Using the condition $(*)$ for the roots $\gamma_i\in\Phi\sm\Delta$ and the absence of components fo type $A_1$ for the roots $\gamma_i\in\Delta$, we can choose such roots $\beta_i\in\Delta$ that $\alpha_i=\gamma_i-\beta_i\in\Phi$, and the corresponding structural constant $c=c_{\alpha_i,\beta_i}$ is invertible in the ring $R$. Then for every factor in the right hand side of the equality $(\#)$ we can express the element $x_{\gamma_i}\big(\tfrac{\xi_i(g)}{s^{N_1}}\big)$ from the Chevalley commutator formula for the commutator $[x_{\alpha_i}(s^{N_1+N_2}\theta_i),x_{\beta_i}(c^{-1}s^{N_1})]$. After that using the Chevalley commutator formula again, we obtain $F_s([g,h])\in F_s(E(\sigma)\cap G(\Phi,R,(s^{N_2})))$. In addition, by assumption we have $g\in G(\Phi,R,(s^{N_2}))$, and by the choice of the number $N_2$ we have $[g,h]\in E(\sigma)$.

In order to finish the proof, it remains to take such $N>N_2$ that $\sigma_\gamma\cap (s^N)\le s^{N_3}\sigma_\gamma$ for any $\gamma\in\Phi$. Such $N$ exists by Artin--Rees lemma.
\end{proof}

\begin{lem}\label{InductionComm} Assume that the subsystem $\Delta$ has no irreducible components of type $A_1,$ and let $\sigma$ be a net of ideals in the ring $R$. Then the following inclusion holds true
$$
[S(\sigma)\cap\tilde{F}_s^{-1}(E(\tilde{F}_s(\sigma))),F_s^{-1}(E(F_s(\sigma)))]\le E(\sigma)\tp
$$
\end{lem}

\begin{proof}
Let $R_i$ be the inductive system of finitely generated subrings of the ring $R$ that contain  the element $s$. Then it is easy to see that the following equalities holds true
\begin{align*}
	S(\sigma)\cap\tilde{F}_s^{-1}(E(\tilde{F}_s(\sigma)))&=\varinjlim (S(\sigma\cap R_i)\cap\tilde{F}_s^{-1}(E(\tilde{F}_s(\sigma\cap R_i)))),
	\\
	F_s^{-1}(E(F_s(\sigma)))&=\varinjlim(F_s^{-1}(E(F_s(\sigma\cap R_i))))\tp
\end{align*}

Therefore, the proof is reduced to the case where $R$ is finitely generates; hence $\tilde{R}_s=\hat{R}_s$.

Let $g\in S(\sigma)\cap\tilde{F}_s^{-1}(E(\tilde{F}_s(\sigma)))$ and $h\in F_s^{-1}(E(F_s(\sigma)))$. Since the ring $R$ is finitely generated and, in particular, Noetherian, we can apply Lemma~\ref{Denominators2} and take such a number $N$ that $[S(\sigma,(s^N)),h]\sub E(\sigma)$.

Since the projection onto quotient ring $\map{\pr}{R}{R/(s^N)}$ factors through homomorphism $\tilde{F}_s$, it follows that $\pr_*(g)\in E(\pr(\sigma))$. Since the group $E(\sigma)$ maps onto the group $E(\pr(\sigma))$ surjectively, it follows that $g=g_1g_2$, where $g_1\in E(\sigma)$ and $g_2\in S(\sigma,(s^N))$. Using the definition of the number $N$ and Theorem~\ref{normal}, we obtain that
\begin{equation*} 
	[g,h]=[g_1g_2,h]=\leftact{g_1}{[g_2,h]}[g_1,h]\in E(\sigma)\tp
	\qedhere
\end{equation*} 
\end{proof}

Lemma~\ref{functor}, and Theorem~\ref{normal} allow us to give the following definition:

$$
G_d(\sigma)=\bigcap_{\substack{\map{f}{R}{A} \\ \delta(A)\le d}}\Ker (G(\sigma)\to G(f(\sigma))/\hat{E}(f(\sigma)))\tp
$$

\begin{lem}\label{GbyGzero} In the previous notation the subgroup $G_0(\sigma)$ is normal in the group $G(\sigma),$ and the corresponding quotient group is abelian.
\end{lem}

\begin{proof}
It is enough to prove that $[G(\sigma),G(\sigma)]\le G_0(\sigma)$, i.e. that for any homomorphism $\map{f}{R}{A}$ to a semilocal ring we have $f_*([G(\sigma),G(\sigma)])\le E(f(\sigma))$. Applying Lemma~\ref{functor}, we may assume that $R=A$ and $f=\id$. In this case the statement follows from Lemma~\ref{GForSemiLocal}.
\end{proof}

\begin{lem}\label{GdbyGdplusone} In the previous notation assume that the subsystem $\Delta$ has no irreducible components of type $A_1$. Then $[G_d(\sigma),G_0(\sigma)]\le G_{d+1}(\sigma)$ for any $d\ge 0$.
\end{lem}
\begin{proof}
We must prove that for any homomorphism
$$
\map{f}{R}{A},\ \text{ where }\ \delta(A)\le d+1,
$$
we have $f_*([G_d(\sigma),G_0(\sigma)])\le E(f(\sigma))$. Applying Lemma~\ref{functor}, we may assume that $R=A$ and $f=\id$. Then there exists a finite collection of maximal ideals $\M_1$, $\dots$, $\M_k$ of the ring $R$ such that for any $s\in R\sm\left(\bigcup_k \M_k \right)$ we have $\delta(\tilde{R}_s)\le d$.

Let $g\in G_d(\sigma)$ and $h\in G_0(\sigma)$. Consider the multiplicative set $S=R\sm\left(\bigcup_k \M_k \right)$. The ring $R[S^{-1}]$ is semilocal; hence by definition of the group $G_0(\sigma)$ there exists an element $s\in S$ such that $F_s(h)\in E(F_s(\sigma))$. In addition by choice of the ideals $\M_1$, $\dots$, $\M_k$ and by definition of the group $G_d(\sigma)$, we have $\tilde{F}_s(g)\in E(\tilde{F}_s(\sigma))$. The statement, therefore, follows from Lemma~\ref{InductionComm}.
\end{proof}

Now we finish the proof of Theorem~\ref{nilpotent}. Assume that the subsystem $\Delta$ has no irreducible components of type $A_1$, and let $\delta(R)=d<\infty$. Then, clearly, we have $G_d(\sigma)\!=\!E(\sigma)$. The statement of the theorem, therefore, follows from lemmas~\ref{GbyGzero}~and~\ref{GdbyGdplusone}.


\end{document}